\documentclass{cccg13}
\usepackage{graphicx,amssymb,amsmath,subfigure,bm, bbold, graphicx, cite, multirow, blkarray, easybmat}

\newcommand{\RR}{\ensuremath{\mathbb{R}}}
\newcommand{\bp}{\ensuremath{\mathbf p}}
\newcommand{\bq}{\ensuremath{\mathbf q}}
\newcommand{\barR}{\ensuremath{\overline{R}}}

\newtheorem{thm}{Theorem}
\newtheorem{defn}[thm]{Definition}
\newtheorem{lem}[thm]{Lemma}

\newtheorem{proposition}[thm]{Proposition}

\title{Hyperbanana Graphs}

\author{Christopher Clement\thanks{Department of Mathematics,
        University of Michigan, {\tt crclement@umich.edu}}
		\and Audrey Lee-St.John\thanks{Department of Computer Science, Mount Holyoke College, {\tt astjohn@mtholyoke.edu}. 
		Research partially supported by the Clare Boothe Luce Foundation.}
		\and Jessica Sidman\thanks{Department of Mathematics \& Statistics, Mount Holyoke College, {\tt jsidman@mtholyoke.edu}}
		\thanks{All three authors were partially supported by NSF grant DMS-0849637.}}

\index{Clement, Christopher}
\index{Lee-St.John, Audrey}
\index{Sidman, Jessica}

\begin{document}
	\setlength\abovedisplayskip{1pt}
		  \setlength\belowdisplayskip{1pt}
		  \setlength\abovedisplayshortskip{1pt}
		  \setlength\belowdisplayshortskip{1pt}
		
\thispagestyle{empty}
\maketitle

\begin{abstract}
	A {\em bar-and-joint} framework is a finite set of points together with specified distances
	between selected pairs.  In rigidity theory we seek to understand when the remaining pairwise
	distances are also fixed.
	If there exists a pair of points which move relative to one another while
	maintaining the given distance constraints, the framework is {\em flexible}; otherwise, it is {\em rigid}. 
	
	Counting conditions due to Maxwell give a necessary combinatorial criterion 
	for generic minimal bar-and-joint rigidity in all dimensions.  Laman showed that these conditions 
	are also sufficient for frameworks in $\RR^2$.
	However, the flexible ``double banana'' shows that Maxwell's conditions are not
	sufficient to guarantee rigidity in $\RR^3.$ 
	We present a generalization of the double banana to a family of {\em hyperbananas}.
	In dimensions 3 and higher, these
	are (infinitesimally) flexible, providing counterexamples to the natural generalization of Laman's theorem.
\end{abstract}

\section{Introduction}
A {\em bar-and-joint framework} is composed of universal {\em joints} whose relative positions are constrained by
fixed-length {\em bars}. An {\em embedding} of such a framework in $\RR^d$ associates a point in $\RR^d$ to each joint   
with the property that the distance between joints connected by a bar is satisfied by the embedding.
Bar-and-joint frameworks can be used to model structures arising in many applications, including sensor networks, 
proteins, and Computer Aided Design (CAD) systems. In combinatorial rigidity theory we seek an understanding of
the structural properties of such a framework, and ask whether it
is  flexible (i.e., admits an internal motion that respects the constraints) or rigid.

  \begin{figure}[t]
\centering
\includegraphics[width=1.25in]{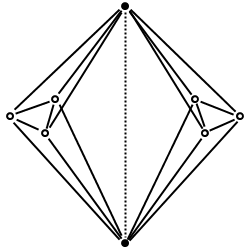}
\caption{The double banana is a Maxwell graph in $\RR^3$, but is flexible. Each ``banana'' can rotate about the
{\em implied hinge} (dotted).}
\label{fig:dblbanana}
\end{figure}

In this paper, we assume that we are given an embedding of a bar-and-joint framework from which
the lengths of bars can be inferred. 
\begin{defn}
A \emph{bar-and-joint framework} $F = (G, \bp)$ 
embedded in $\mathbb{R}^d$ 
is composed of a graph $G = (V, E)$ with $|V| = n$ and $|E|=m$
and an embedding $\bp: V \rightarrow \RR^d$, which assigns a 
position vector $\bp_i$ to each vertex $v_i$.
\end{defn}
We only concern ourselves with \textit{generic embeddings} of these frameworks, which
can be thought of as embeddings with the properties we would expect if we chose an embedding at random.  To formally define
genericity we require the notion of a {\bf rigidity matrix}, which encodes the infinitesimal behavior
of the framework.
\begin{defn} \label{defn:rm}
For a framework $F = (G,\bp)$ embedded in $\mathbb{R}^d$ we define a \emph{rigidity matrix} $M_F$ to be an $m \times dn$ matrix in which
the columns are grouped into $n$ sets of $d$ coordinates for each vertex.
Each row of the rigidity matrix corresponds to an edge $ij$ and has the following pattern.
{\scriptsize
\[
  \begin{blockarray}{cccccccc}
    & v_1 & \dots & v_i & \cdots & v_j & \dots & v_n \\
    \begin{block}{c[ccccccc]}
    \hspace{-2mm}ij & 0 & \cdots 0 \cdots & \bp_i - \bp_j & \cdots 0 \cdots & \bp_j - \bp_i & \cdots 0 \cdots & 0 \\
    \end{block}
  \end{blockarray}
\]
}
\end{defn}

If $F$ is a framework, $M_F$ determines if it is {\em infinitesimally} flexible or rigid; for brevity,
we omit ``infinitesimally'' for the remainder of this paper.
We say that $F$ is \textit{rigid} if the insertion of any new bar between vertices
does not change the rank of $M_F$;
otherwise it is \textit{flexible}.  
A rigid framework is \textit{minimally rigid} if the rows of $M_F$ are independent.

The \textit{infinitesimal motions} of $F$ can be encoded 
by assigning a velocity vector $\bp_i' \in \RR^d$ to each vertex $v_i$ 
so that $(\bp_1', \dots, \bp_n')$ is nonzero and is in the null space of $M_F$
(intuitively, these are instantaneous velocities that
do not shrink or stretch the bar constraints).  
There is always a set of {\em trivial motions} corresponding to rigid body motions of
$\RR^d$; the
space of rigid motions of $\RR^d$ has dimension $\binom{d+1}{2}$ and is 
generated by rotations about $(d-2)$-dimensional affine linear subspaces and translations.
In general, then, a framework on at least $d$ vertices is minimally rigid if and only
if its rigidity matrix has nullity $\binom{d+1}{2}$.
However, if a framework $F$ is contained in an affine subspace $H \subset \RR^d$ where $\dim H \leq d-2,$ 
then there is a rigid motion of $\RR^d$ that fixes $F$; hence,
the null space of $M_F$ has dimension less than $\binom{d+1}{2}$.

Combinatorial counting conditions, first observed by Maxwell \cite{maxwell1864}, give a necessary condition 
for minimal bar-and-joint rigidity. 
Throughout this paper, we will use the convention that,
if $V'$ is a subset of the vertices of a graph $G$ and $\mathcal{E}$ is a subset of the edges of $G,$ 
then $\mathcal{E}(V')$ is the set of edges in $\mathcal{E}$ induced by 
the vertices in $V'.$  

\begin{defn}
	A {\em Maxwell graph} $G=(V,E)$ in dimension $d$ satisfies
\begin{enumerate}
\item \label{maxwell:total} $|E| = d|V| - \binom{d+1}{2}$
\item \label{maxwell:sub} $|E(V')| \leq d|V'| - \binom{d+1}{2},$ for all $V' \subseteq V$ where $|V'| \geq d$.
\end{enumerate}
\label{maxwell}
\end{defn}

For almost all frameworks $F=(G, \bp)$ on a fixed graph $G$, the rank of $M_F$ is constant, as 
the set of special embeddings for which $M_F$ drops rank is parameterized by a closed subset of $\RR^{dn}.$
We formally define genericity as follows.

\begin{defn}
A framework $(G,\bp)$ is {\em generic} if its rigidity matrix achieves the maximum rank over
all frameworks $(G,\bq)$. 
\end{defn}
We call a framework {\em generically minimally rigid} if there exists a generic framework with the
same underlying graph that is minimally rigid.  
We analyze the generic behavior of a framework purely by the combinatorial structure of the graph.
Therefore, from here on we will write $M_G$ to denote the rigidity matrix associated to a generic embedding
 of $G.$

In $\RR^2$, Laman proved that the Maxwell conditions are sufficient for generic minimal rigidity.
\begin{thm}[Laman\cite{laman:rigidity:1970}]
A bar-and-joint framework, with underlying graph $G = (V, E)$, 
embedded in $\mathbb{R}^2$ is generically minimally rigid if and only if it satisfies the following conditions:
\begin{enumerate}
\item $|E| = 2|V| - 3$
\item $|E(V')| \leq 2|V'| - 3,$ for all $V' \subseteq V$ where $|V'| \geq 2$
\end{enumerate}
\label{laman}
\end{thm}
However, the sufficiency of the Maxwell counting conditions for rigidity
does not generalize to higher dimensions.  In $\RR^3$, the
well-known ``double banana'' is a Maxwell
graph that is flexible \cite{structrigidityDblBanana}. This structure is composed of 
two ``bananas'' joined on a pair of vertices (refer to Figure \ref{fig:dblbanana}) and exhibits a hinge
motion
about the dotted line.
This denotes the existence of an \textit{implied edge} between two vertices that are not incident to each other,
yet whose distance is fixed as a consequence of the other constraints. Since a rotation is allowed about the edge, it is 
called an \textit{implied hinge}.

Counterexamples like the double banana can provide insight into the challenges
presented in dimension 3 and higher for which no combinatorial characterization of bar-and-joint rigidity is known. 

\medskip\noindent{\bf Contributions.}
In this paper, we describe a class of graphs called {\em hyperbananas}
that generalize the double banana to higher dimensions.  We present hyperbananas that are
Maxwell graphs and show these to be (infinitesimally) flexible.
To the best of our knowledge, this is the first family of counterexamples to the sufficiency of the
Maxwell conditions for minimal bar-and-joint rigidity 
addressing
all dimensions of 3 and higher.

\medskip\noindent{\bf Related work.}
Other generalizations of the double banana include
the banana spider graphs of Mantler and Snoeyink  \cite{mantlerSnoeyinkCCCG}.
These were developed to address 
an attempt at classifying
3D bar-and-joint rigidity by vertex connectivity, 
as it was conjectured that all graphs with implied hinges must be 2-connected (like the double banana). 
The banana spider graphs provide examples with higher vertex connectivity, answering this conjecture in the negative.
The key idea was to add ``spider" components to the double banana,
 increasing vertex connectivity while maintaining flexibility about the implied hinge.

Another class of counterexamples to Maxwell's conditions in 3D was developed by Cheng et al. \cite{chengSitharamStreinuCCCG}.
These ``ring of roofs'' frameworks, first described by Tay \cite{ringOfRoofs}, 
provide examples of flexible Maxwell graphs that admit
no non-trivial rigid subgraphs, i.e., rigid subgraphs larger than a tetrahedron.  
This countered an earlier attempt by Sitharam and Zhou \cite{sitharamZhou04}
to characterize 3D bar-and-joint rigidity by detecting rigid components and adding the resulting implied edges.

\section{Maxwell hyperbananas}
We now present a family of graphs called hyperbanana graphs; under certain
conditions, hyperbananas are Maxwell graphs. We generalize 
the double banana, which consists of two minimally rigid ``bananas'' glued
together on a pair of vertices. Each banana can be built using the following inductive 
construction.
\begin{defn}
Fix a positive integer $d.$   A \emph{d-Henneberg 0-extension} on a 
graph $G$ 
results in a new graph by adding a single vertex and connecting it to $d$ 
distinct vertices in $G$. 
\end{defn}
When a $d$-Henneberg 0-extension is applied to a minimally rigid framework in $\RR^d$,
 minimal rigidity is preserved, and hence so are the
Maxwell conditions \cite{taywhiteley}. In the double banana, each individual banana is created by two 3-Henneberg 0-extensions on a triangle
(which is minimally rigid in $\RR^3$), connecting each new vertex
to the 3 vertices of the triangle.

Before generalizing the banana construction, we give some additional notation.  If $U$ and $W$ are finite sets,
let $K_U$ denote the complete graph with vertex set $U$ and $K_{U,W}$ be the complete bipartite graph on the 
two disjoint sets $U$ and $W$.
\begin{defn}
A {\em banana bunch} is a graph $B_{d,b}$ 
obtained by performing $b$ $d$-Henneberg 0-extensions on a $K_d.$
The $b$ vertices added by the Henneberg extensions are called
{\em banana vertices}.
\end{defn}
Since $K_d$ embedded in $\RR^d$ is minimally rigid for any $d$, 
$B_{d,b}$ is generically minimally rigid in dimension $d$.

Hyperbananas are composed of two banana bunches glued together along the banana vertices.  

\begin{defn}
For $i = 1,2,$ let $G_i$ be a copy of $B_{d,b}$ with vertex set partitioned into $V_i \cup U_i$, where the $K_d$ has vertex set $V_i$ and the set $U_i$ consists of banana vertices.  We define
the \emph{hyperbanana} $H_{d,b}$ to be $G_1 \cup G_2/\sim$, where $\sim$ identifies banana vertices based on some fixed bijection from $U_1$ to $U_2.$  The 
vertex set of $H_{d,b}$ is the set $V = V_1\cup V_2 \cup U,$ where $U$ is the set of banana vertices.
\end{defn}

The double banana is simply $H_{3,2}$.
An example of a higher dimensional hyperbanana, $H_{5,3}$, is pictured in Figure \ref{fig:kb53}.
While this is a Maxwell graph, not all choices of $b$ and $d$ satisfy the counting conditions.
\begin{figure}[t]
\subfigure[]{	\begin{minipage}[b]{0.48\linewidth}
	\centering\includegraphics[scale=.5]{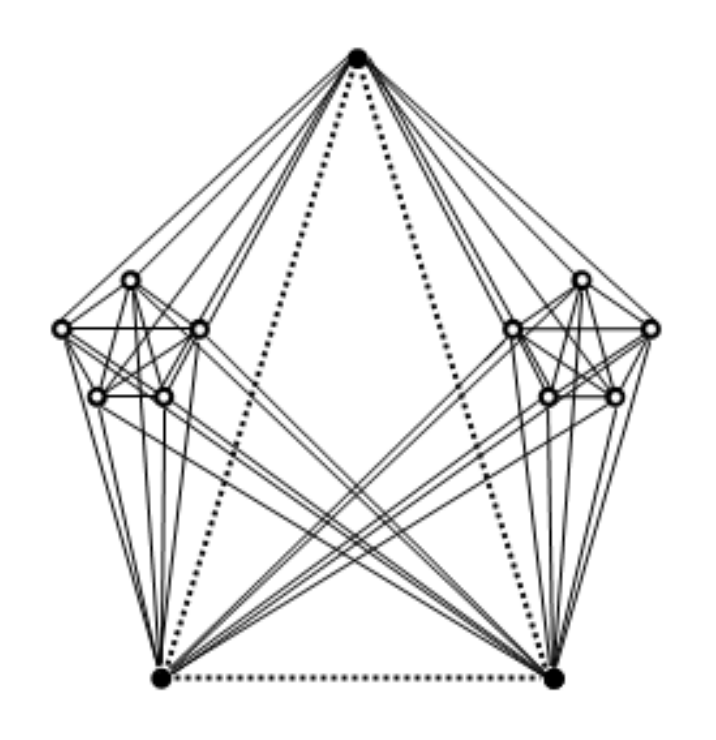}\end{minipage}\label{fig:h53}}
	\hfil
   \subfigure[]{	\begin{minipage}[b]{0.48\linewidth}
	\centering\includegraphics[scale=.5]{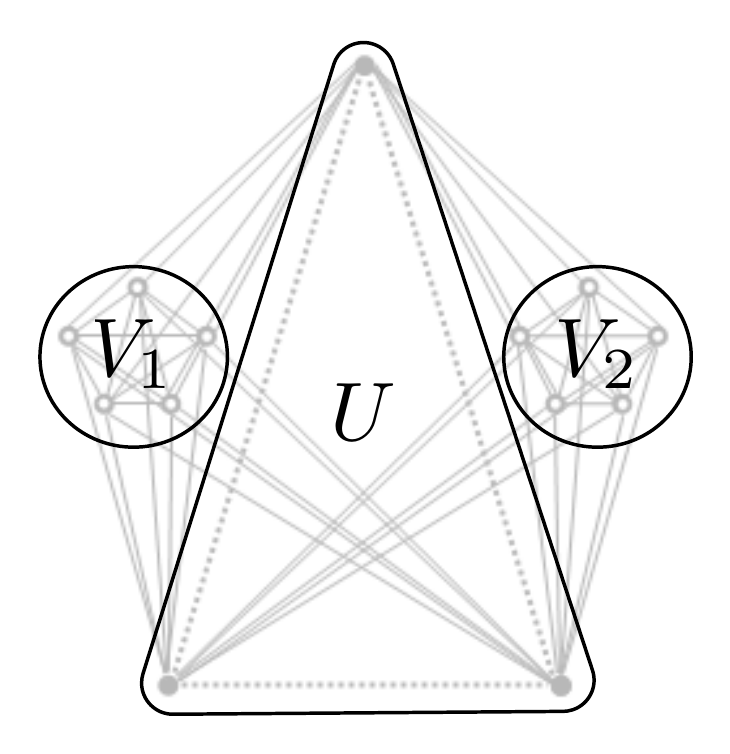}\end{minipage}\label{fig:h53-labeled}}
\caption{The hyperbanana $H_{5,3}$ is a flexible Maxwell graph $\mathbb{R}^5$; there are 3
implied edges (dotted) among the vertices in $U$.}
\label{fig:kb53}
\end{figure}
For example, simply checking the counts on the total number of edges for the hyperbanana $H_{4,3}$ confirms 
that this graph has too many edges to be Maxwell. 
In fact, it is rigid in $\mathbb{R}^4$, but {\em overconstrained}.  Therefore, it is not minimally rigid as its
rigidity matrix contains dependencies.  
Checking the counts on the total number of edges for the hyperbanana $H_{6,3}$ shows that it is {\em underconstrained} and
therefore flexible in $\mathbb{R}^6$.  

\subsection{Odd-dimensional hyperbananas}
When $d$ is odd and equal to $2b-1$, we obtain hyperbananas that are Maxwell graphs.
We begin with a more general lemma that will be used in proving the counting conditions.
In the proofs that follow, 
we define
$V_i'  = V' \cap V_i$ and $U' = V' \cap U$
for a subset $V'$ of the vertex set of $H_{d,b}$, 

\begin{lem}
	\label{lem:genIneq}
If $H_{d,b}=(V,E)$ 
and $V' \subseteq V,$ and $|V_i' \cup U'| \geq d$ for $i=1,2,$ then \[
|E(V')| \leq d|V'| - 2\tbinom{d+1}{2} + d|U'|.
\]
\end{lem}
\begin{proof}
As each banana bunch is minimally rigid we have
\begin{equation}
		\label{eq:case1}
		|E(V_i' \cup U')| \leq d|V_i' \cup U'| - \tbinom{d+1}{2}
		\end{equation} 
	for each $i.$
	Adding the inequalities yields
	\begin{equation}\label{eq:case2}
	\begin{split}
	|E(V')| \leq d(|V_1'| + |V_2'| + 2|U'|) - 2\tbinom{d+1}{2}\\
	= d(|V_1'| + |V_2'| + |U'|) - 2\tbinom{d+1}{2} +d|U'|\\
	 =d|V'| - 2\tbinom{d+1}{2} +d|U'|.\\
	 	\end{split}
	\end{equation}
\end{proof}

We can now show that the specific class of hyperbananas in odd-dimensional spaces are Maxwell graphs.
\begin{thm}
	\label{thm:oddHyperbananas}
The hyperbanana $H_{d,b}$ embedded in $\mathbb{R}^{d}$ with $d = 2b-1$ is a Maxwell graph.
\end{thm}

\begin{proof}
We check condition \ref{maxwell:total} of Definition \ref{maxwell} by vertex and edge counts.  
The graph $H_{d,b}$ has $d$ vertices from each complete $K_d$ graph
and $b$ banana vertices, totaling $2d + b$ vertices.  Since $d = 2b - 1$, there are $\frac{5d+1}{2}$ vertices.  
Each $K_d$ has $\binom{d}{2}$ edges, and each banana vertex is incident to $2d$ edges.  
This sums to an edge count of $2\binom{d}{2} + 2d(\frac{d+1}{2})$.  Simplifying, we can verify that the edge count is $|E| = 2d^2.$
Substituting the vertex count $|V| = \frac{5d+1}{2}$, we see that Maxwell condition \ref{maxwell:total} is satisfied:
$$d|V| - \tbinom{d+1}{2} = d\left(\tfrac{5d+1}{2}\right) - \tbinom{d+1}{2} = |E|$$

Now we check Maxwell condition \ref{maxwell:sub}. If $V'$ is contained within a single banana 
bunch, the condition is satisfied as $B_{d,b}$ is minimally rigid and therefore Maxwell.
If $V'$ intersects both banana bunches non-trivially, then
there are three cases which depend on whether the intersection 
with each banana bunch 
contains at least $d$ vertices.

If $|V_i'\cup U'| \geq d$ for both $i$,  then combining  $|U'| \leq b = \frac{d+1}{2}$ with Lemma \ref{lem:genIneq} gives the result.

Now suppose, without loss of generality, that $|V_1' \cup U'| \geq d,$ but
	$|V_2' \cup U'| < d.$  We know that
	\begin{align}
	|E(V_2' \cup U')| &= \tbinom{|V_2'|}{2} + |U'||V_2'|\\
		&= \frac{(|V_2'|-1)|V_2'|}{2} +  |U'||V_2'|\\
		&\leq\frac{(d-2)|V_2'|}{2} +  b|V_2'|\label{eq:b}
		\end{align}
		Since $b=\frac{d+1}{2}$, we obtain $|E(V_2' \cup U')| \leq d|V_2'|$.
	Combining this with Inequality \ref{eq:case1} gives the desired inequality in the second case.
	
	Finally, suppose that both $|V_i' \cup U'| < d$ and $|V_1'| \geq |V_2'|.$  As $|V_1' \cup V_2' \cup U'| \geq d,$ there
	exists a subset $W \subseteq V_2'$ so that $|V_1' \cup W \cup U'|=d$. Let $W' = V_2' \backslash W.$  The set $|E(V_2')|$ 
	consists of the edges of $K_W,$ the edges of $K_{W'}$ and the edges of $K_{W,W'}.$

	Now suppose we had a set $V_1''$ satisfying $V_1 \supseteq V_1'' \supset V_1'$ and $|V_1'' \cup U'|=d.$  Then 
	\[|E(V_1' \cup W \cup U')|+|E(K_{V_1',W})| = | E(V_1''\cup U') |,\]  and by Inequality \ref{eq:case1},
	\begin{equation}
	\label{eq:case3}
	|E(V_1' \cup W \cup U')|+|E(K_{V_1',W})|\leq d|V_1' \cup W \cup U'| - \tbinom{d+1}{2}.
	\end{equation}
	Applying the argument in the second case to the set $W' \cup U'$ and adding the inequality to \ref{eq:case3}, gives
	 the result in this final case as $|E(K_{W,W'})| < |E(K_{V_1',W})|$.
	
\end{proof}

\subsection{Even-dimensional hyperbananas}
We observed earlier that hyperbananas may be either overconstrained or underconstrained 
in even-dimensional spaces and are not Maxwell graphs.  
However, by making a small modification to our definition, 
we 
obtain Maxwell graphs for even-dimensional spaces.

\begin{defn}
For  even $d$, we define the \emph{even hyperbanana} to be a graph $H^+_{d,b}$ consisting of a hyperbanana $H_{d,b}$ together with an additional $\frac{d}{2}$ edges connecting distinct vertices of the complete graphs in the two banana bunches.
\end{defn}

This addition of $\frac{d}{2}$ edges between the complete graphs in $H_{d,b}$ 
 results in
$H^+_{d,b}$ being a Maxwell graph for the even-dimensional spaces for certain values of $d$ relative to $b$.
One example of an even hyperbanana, $H^+_{4,2}$, is shown in Figure \ref{fig:kbplus42}.
Note that $H^+_{d,b}=(V,F)$, is built from $H_{d,b}=(V,E)$; let $E^+$ be the additional
$\frac{d}{2}$ edges so that $F = E \cup E^+$. In Figure \ref{fig:kbplus42}, for example,
$E^+$ is composed of the 2 dashed edges. 

\begin{figure}[tp]
\centering
\includegraphics[width=1.6in]{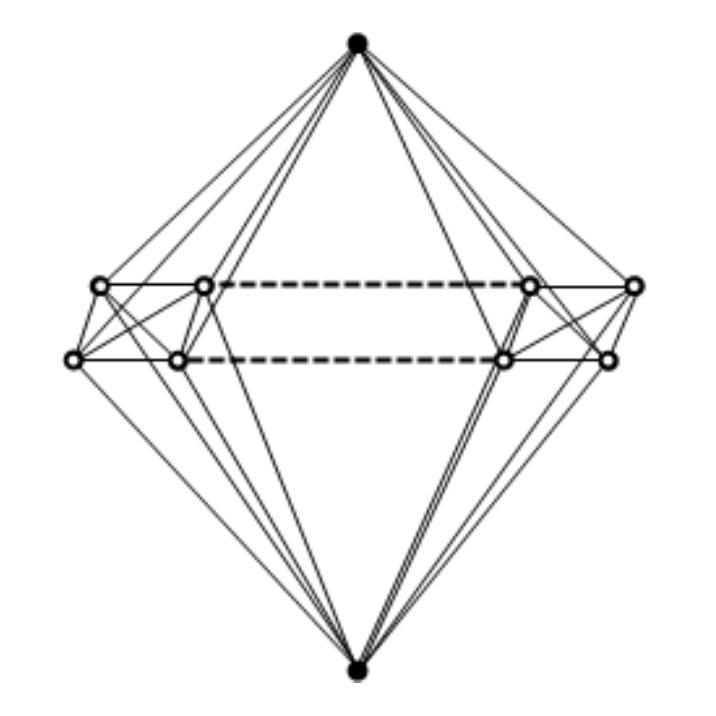}
\caption{The even hyperbanana $H^+_{4,2}$ is a flexible Maxwell graph; it is built from
the hyperbanana $H_{4,2}$ by an additional 2 (dashed) edges.}
\label{fig:kbplus42}
\end{figure}

\begin{thm}
The even hyperbanana $H^+_{d,b}=(V,F)$ embedded in $\mathbb{R}^{d}$ with $d = 2b$ is a Maxwell graph.
\end{thm}

\begin{proof}
Since $d=2b$, the number of vertices in $H^+_{d,b}$ is $|V| = 2d + b = \frac{5}{2}d$, as there are two $K_d$ graphs and $b$ banana vertices.  
There are 2 complete graphs with $\binom{d}{2}$ edges, $b$ banana vertices connecting to the $2d$ complete graph vertices, and $\frac{d}{2}$ edges between the complete graphs, resulting in $|F| = 2d^2 - \frac{d}{2}$.
By substituting the vertex count, we can verify Maxwell condition \ref{maxwell:total}.  
$$d|V| - \tbinom{d+1}{2} = 2d^2 - \frac{d}{2} = |F|.$$

Now let $V' \subseteq V$ with $|V'| \geq d$.  If $V'$ is completely contained
in a banana bunch, Maxwell condition \ref{maxwell:sub} is satisfied as $B_{d,b}$ is 
Maxwell. Assume, then, that $V'$ non-trivially intersects
both vertex sets $V_1$ and $V_2$.

If $|V_i'\cup U'| \geq d$ for both $i = 1,2,$ then by Lemma \ref{lem:genIneq},
$$|E(V')| \leq d|V'| -2 \tbinom{d+1}{2} + d|U'|).$$
The number of banana vertices is $b = \frac{d}{2}$, so $|U'| \leq \frac{d}{2}$. Therefore,
\[
\begin{split}
|E(V')| &\leq d|V'| -2 \tbinom{d+1}{2}+ \frac{d^2}{2}\\
& =  d|V'| - \tbinom{d+1}{2} - \frac{d^2+d}{2}+ \frac{d^2}{2}\\
&=d|V'| - \tbinom{d+1}{2} - \frac{d}{2}.
\end{split}
\]

Since $F = E \cup E^+$, $|F(V')| = |E(V')| + |E^+(V')|$. By adding $|E^+(V')|$ to 
both sides of the previous inequality we obtain
$$|F(V')| \leq d|V'| - \tbinom{d+1}{2} - \frac{d}{2} + |E^+(V')|.$$
By definition, $|E^+| = \frac{d}{2}$, implying $|E^+(V')| \leq \frac{d}{2}$. 
Therefore, we can conclude that Maxwell condition \ref{maxwell:sub}, 
$$|F(V')| \leq d(|V'|) - \tbinom{d+1}{2},$$
holds in this case.

Now suppose, without loss of generality, that $|V_1' \cup U'| \geq d,$ but
	$|V_2' \cup U'| < d.$ 
	Since $b = \frac{d}{2}$, Inequality \ref{eq:b} implies
	\begin{equation}\label{eq:d-1}
		|E(V_2'\cup U')| \leq (d-1)|V_2'|. 
	\end{equation}
	We can combine this
	with 
	\[	|E(V_1'\cup U')| \leq d|V_1'\cup U' | - \tbinom{d+1}{2}\]	
and the edges in $E^+(V')$ to obtain
\[\begin{split}
|F(V')| &\leq d|V_1'\cup U' | - \tbinom{d+1}{2}+(d-1)|V_2'| +|E^+(V')|\\
&= d|V'| -  \tbinom{d+1}{2} - |V_2'| +|E^+(V')|\\
& \leq d|V'| -  \tbinom{d+1}{2} 
\end{split}
\]
as $|E^+(V')| \leq |V_2'|.$

Finally, suppose that both $|V_i' \cup U'| <d.$  Assume that $|V_1'| \geq |V_2'|$ and define $W$ and $W'$ as in the proof of Theorem
 \ref{thm:oddHyperbananas}.
Adding Inequalities \ref{eq:case3} and \ref{eq:d-1} (with $W'$ replacing $V_2'$), 
\begin{align*}
&|E(V_1' \cup W \cup U')|+|E(K_{V_1',W})| +|E(W' \cup U')| \\
& \leq d|V_1' \cup W \cup U'| - \tbinom{d+1}{2}+(d-1)|W'|\\
& = d|V'|- \tbinom{d+1}{2}-|W'|,
\end{align*}
and hence 
\begin{align*}
&|E(V_1' \cup W \cup U')|+|E(K_{V_1',W})| +|E(W' \cup U')|+|W'| \\
&\leq d|V'|- \tbinom{d+1}{2}.
\end{align*}

Since $|F(V')|$ is equal to
\[
|E(V_1' \cup W \cup U')|+|E(K_{W,W'})| +|E(W' \cup U')|+|E^+(V')|,
\]
it will suffice to show that
\[
|E(K_{W,W'})| +|E^+(V')| \leq|E(K_{V_1',W})|+|W'|, 
\]
or that
\begin{equation}
\label{eq:even}
|W| \cdot |W'| +|E^+(V')| \leq |V_1'|\cdot |W|+|W'|
\end{equation}

Now let $t = |V_1'| - |W'|$.  Since $|V_1'| \geq |V_2'|$ and $|V_1'| < d,$ $|W|>0$, which implies that $|V_1'| > |W'|$ and hence that $t\geq 1$.
Setting $|W'| = |V_1'| - t$, we have
\begin{align*}
&|W| \cdot |W'| +|E^+(V')|\\
& =|W|\cdot(|V_1'| - t) + |E^+(V')|\\
&=|V_1'|\cdot|W| - t|W| + |E^+(V')| \\
&\leq |V_1'|\cdot|W| -|W| + |E^+(V')|,
 \end{align*}
 as $t \geq 1.$
Then 
\[
|V_1'|\cdot|W| -|W| + |E^+(V')| \leq |V_1'|\cdot |W|+|W'|
\]
if and only if
\[
|V_1'|\cdot|W| + |E^+(V')| \leq |V_1'|\cdot |W|+|W'|+|W|.
\]
Indeed, since $|W'|+|W| = |V_2'| \geq  |E^+(V')|$, this inequality
holds, completing the proof.
\end{proof}

\section{Flexible hyperbananas}
In this section, we prove that the Maxwell hyperbananas are flexible.

We begin by considering the rigidity matrix $M_{B_{d,b}}$ for a generic framework on the banana bunch $B_{d,b}$ in dimension $d$, 
which has $d(d+b)$ columns and $\binom{d}{2} + db$ rows.  Since the banana bunch is minimally rigid, the rank of its rigidity matrix is maximal
and equal to the number of rows $\binom{d}{2} + db$.  
Let the vertex set of $B_{d,b}$ be partitioned into sets $V_1$ and $U$, where the set $U$ consists of banana vertices.
Assume that the columns of $M_{B_{d,b}}$ are arranged so that the columns corresponding to the vertices in $V_1$ come first, followed by the columns for $U$.%
\begin{lem}
Each row of the block matrix 
\[
\bigg[\begin{array}{c |c}   
     0 & M_{K_U} 
 \end{array}
 \bigg]\]
with $d^2$ columns of zeros ($d$ columns for each vertex in the $V_1$), is in the row space of $M_{B_{d,b}}.$
\label{dependence}
\end{lem}

\begin{proof} 
Since the banana bunch is minimally rigid and spans $\RR^d$, $M_{B_{d,b}}$ has nullity $\binom{d+1}{2}$.  
If we add an edge from $K_U$, the new rigidity matrix will still have nullity $\binom{d+1}{2}.$    
Thus, each such row must be a linear combination of the rows of $M_{B_{d,b}}$.
\end{proof}

\begin{proposition}\label{prop:reduce}
If $B_{d,b} = (V_1 \cup U, E)$ is embedded in $\RR^d,$ and the rank of $M_{K_U}$ is $\binom{b}{2},$ 
then $M_{B_{d,b}}$ is row-equivalent to a matrix of the form 
\[
  \begin{blockarray}{cc}
      \begin{block}{[c|c]}
     \BAmulticolumn{2}{c}{M_{B_{d,b}}^*} \\
    \cline{1-2}
     0 & M_{K_U} \\
    \end{block}
  \end{blockarray}
,\]
where $M_{B_{d,b}}^*$ 
consists of $|E|-\binom{b}{2}$ rows of the original matrix $M_{B_{d,b}}$.
\end{proposition}

\begin{proof}
Let $R$ be a row in $[\begin{array}{c|c}   
    0 & M_{K_U} 
 \end{array}].$  By Lemma \ref{dependence}, $R$ may be written as a linear combination of rows of $M_{B_{d,b}}.$  Any row of $M_{B_{d,b}}$ appearing in such a linear combination with a nonzero coefficient may be replaced by $R$ through a sequence of elementary row operations.  Any subsequent row $R'$ of $[\begin{array}{c|c}   
    0 & M_{K_U}
 \end{array}]$ will remain dependent on the rows of the modified matrix.  Moreover, when we express $R'$ as a linear combination of the current set of rows, some  
 remaining row of the original matrix $M_{B_{d,b}}$ must appear with a nonzero coefficient as the rows of $M_{K_U}$ are independent.  Thus, we can insert each row of $[\begin{array}{c|c}   
    0 & M_{K_U}
 \end{array}]$ in this way.

\end{proof}

With this we can prove the following theorem.

\begin{thm}
If $G$ is the hyperbanana $H_{d,b} \subset \mathbb{R}^d$ where $d = 2b-1$ or 
$H_{d,b}^+ \subset \RR^d$ where $d=2b$ and $b \geq 2$,
 then $G$ is flexible. 
\end{thm}
\begin{proof}
Consider the hyperbanana $H_{d,b}$ partitioned into two bunches $B_{d,b}(1)$ and $B_{d,b}(2)$.  Let $M_{B_{d,b}}(1)$ be the rigidity matrix for $B_{d,b}(1)$, $M_{B_{d,b}}(2)$ be the rigidity matrix for $B_{d,b}(2)$ and $M$ be the rigidity matrix for $H_{d,b}$.  If we put the vertices in an order with $(V_1, U, V_2)$ and order the columns of $M$ accordingly, then $M$ is a block matrix of the form 

\[ 
  \begin{blockarray}{cccc}
    & V_1 & U & V_2 \\
    \begin{block}{c[ccc]}
	    B_{d,b} (1) &  \BAmulticolumn{2}{c|}{M_{B_{d,b}}(1)} &0\\
    \cline{1-4}
    B_{d,b} (2) &0 & \BAmulticolumn{2}{|c}{M_{B_{d,b}}(2)} \\
    \end{block}
  \end{blockarray}
\]

By Proposition \ref{prop:reduce} $M$ is row equivalent to

\[ 
  \begin{blockarray}{cccc}
    & V_1& U & V_2 \\
    \begin{block}{c[c|c|c]}
	    \BAmulticolumn{1}{c}{\multirow{2}{*}{$B_{d,b} (1)$}}& \BAmulticolumn{2}{c|}{M_{B_{d,b}}(1)^*} & \BAmulticolumn{1}{c}{\multirow{2}{*}{$0$}}\\
	 \cline{2-3}
	     & 0& M_{K_U}& \\
    \cline{1-4}
      \BAmulticolumn{1}{c}{\multirow{2}{*}{$B_{d,b} (2)$}} & \BAmulticolumn{1}{c}{\multirow{2}{*}{$0$}} &\BAmulticolumn{2}{|c}{M_{B_{d,b}}(2)^*}  \\
    \cline{3-4}
    & & M_{K_U}& 0 \\
    \end{block}
  \end{blockarray}
\]

We can see that there are at least $\binom{b}{2}$ dependencies in $M$, since the $[0|M_{K_U}|0]$ is seen twice in the matrix.
Therefore, since the number of columns is $d|V|$ and the number of
rows is $|E|$, the nullity of $M$ is at least $\binom{d+1}{2}+\binom{b}{2}.$
Thus, since a framework with at least $d$ vertices is minimally rigid in $\RR^d$
if and only if it has nullity $\binom{d+1}{2}$, $H_{d,b}$ is flexible.
Moreover, since $M$ is a submatrix of the rigidity matrix of $H^+_{d,b}$, which satisfies
the Maxwell counts, we see that
$H^+_{d,b}$ is also flexible.
\end{proof}

For odd-dimensional bananas, we can show this bound is tight using the following proposition.
\begin{proposition}\label{banana}
Any linear combination of rows of $M_{B_{d,b}}^*$ of the form 
\[
  \begin{blockarray}{cc}
     V_1 & U \\
    \begin{block}{[c|c]} 
     0 & *  \\
    \end{block}
  \end{blockarray},
\]
 must be trivial, where the $*$ represents potentially nonzero entries.
\end{proposition}

\begin{proof}
Suppose for contradiction that there is a linear combination of rows of $M_{B_{d,b}}^*$ equal to $R$ where 
$R$ has nonzero entries only in columns corresponding to $U$.  Let $\barR$ be the projection of $R$ 
to the columns corresponding to $U.$

If $\barR$ is dependent on the rows in $M_{K_U}$, then the rank of $M_{B_{d,b}}$ is not maximal, which is a contradiction.  
So, we must assume that $\barR$ is independent of these rows.  Thus, the nullspace of $M_{K_U}$ augmented by the row $\barR$
is smaller than the nullspace of $M_{K_U}$.  But all of the elements of the nullspace of $M_{K_U}$ are obtained from rigid motions
of $\RR^d$.  So there is a nonzero vector $\bp' \in \RR^{db}$ in the null space of $M_{K_U}$ which assigns velocities 
to vertices in $U$ 
and has the property that
$\barR \cdot \bp' \neq 0.$  

Since $K_U$ is rigid, $\bp'$ must be obtained by restricting a rigid motion of $\RR^d$ to $K_U.$  Applying
this rigid motion to all of $B_{d,b}$ gives a vector $\bq'$ that assigns velocities to all vertices in $B_{d,b}$ and is equal to $\bp'$ for the vertices in $U.$
As $R$ has nonzero entries only in columns corresponding to $U,$ $\bq' \cdot R = \bp' \cdot R \neq 0.$  $R$
is in the row space of $M_{B_{d,b}}$, so this implies that the nullspace of $M_{B_{d,b}}$ is missing one of the rigid motions of $\RR^d$. 
This is a contradiction because $M_{B_{d,b}}$ is a rigidity matrix.
\end{proof}

\begin{thm}
The hyperbanana $H_{d,b} \subset \mathbb{R}^d$ where $d = 2b-1$ has rigidity matrix $M_{H_{d,b}}$ with nullity exactly $\binom{d+1}{2}+\binom{b}{2}.$  
\end{thm}
\begin{proof}
We will  show that  
\[ M'=
  \begin{blockarray}{cccc}
    & V_1 & U& V_2 \\
    \begin{block}{c[c|c|c]}
	    \BAmulticolumn{1}{c}{\multirow{2}{*}{$B_{d,b} (1)$}}& \BAmulticolumn{2}{c|}{M_{B_{d,b}}(1)^*} & \BAmulticolumn{1}{c}{\multirow{2}{*}{$0$}}\\
	 \cline{2-3}
	     & 0& M_{K_U}& \\
    \cline{1-4}
	\BAmulticolumn{1}{c}{\multirow{1}{*}{$B_{d,b} (2)$}} & \BAmulticolumn{1}{c|}{\multirow{1}{*}{$0$}}& \BAmulticolumn{2}{c}{M_{B_{d,b}}(2)^*}\\
    \end{block}
  \end{blockarray}
\]
has full rank and hence nullity $\binom{d+1}{2}+\binom{b}{2}$.

Since $M_{B_{d,b}}(1)$ has full rank, we know that the top block of $M'$
has linearly independent rows.  
Similarly, the rows in $[0|M_{B_{d,b}}(2)^*]$ are also an independent set. 

Now suppose there is a row $R \in [0|M_{B_{d,b}}(2)^*]$ that is dependent on the upper block of $M'$; 
then $R$ is a linear combination of the rows of $[M_{B_{d,b}}(1)^*|0]$ and
$[0|M_{K_d}|0]$. 
There must be at least one row of $[M_{B_{d,b}}(1)^*|0]$ with a nonzero coefficient or 
we would contradict the independence of $ [0|M_{B_{d,b}}(2)]$.  
Since $R$ is zero in the columns corresponding to vertices in $V_1,$ this 
implies that there is a linear combination of rows of  $[M_{B_{d,b}}(1)^*|0]$ that is 
nonzero only in the banana vertex columns, which contradicts Proposition \ref{banana}.

\end{proof}

\section{Conclusions and Future Work}
We presented a family of hyperbanana graphs and showed that they are Maxwell graphs
under certain conditions. We further
proved that they are flexible, providing counterexamples to the sufficiency 
of the Maxwell counts for bar-and-joint rigidity in dimensions 3 and higher. 

For hyperbananas embedded in odd-dimensional spaces, we
gave a precise analysis of the space of infinitesimal motions.
However, it remains an open problem to give an exact analysis for the even hyperbananas, 
as the addition of the $\frac{d}{2}$ edges prevents us from extending our proof.
Based on Mathematica calculations on randomized embeddings of even hyperbananas, 
we conjecture the following:
\begin{conj}
The even hyperbanana $H^+_{d,b} \in \mathbb{R}^d$ where $d = 2b$ and $b \geq 2$ 
has a rigidity matrix with nullity exactly $\binom{d+1}{2} + \binom{b}{2}$.
\end{conj}

Since counterexamples provide an increased understanding of barriers to finding combinatorial
characterizations of higher-dimensional bar-and-joint rigidity,
it would also be interesting to further generalize the hyperbananas by parametrizing the number
of banana bunches instead of always gluing two.

\scriptsize
\bibliographystyle{plain}

\end{document}